\documentclass[12pt,reqno]{amsart}

\usepackage{amsmath , amssymb, latexsym }
\usepackage{graphics}
\usepackage{graphicx}
\usepackage{verbatim}
\usepackage{enumerate}

\usepackage[normalem]{ulem}

\usepackage[mathscr]{eucal}

\usepackage{amsthm}
\usepackage{mathrsfs}
\usepackage{amsfonts}
\usepackage{paralist} 
\usepackage{xcolor}
\usepackage{hyperref}
\usepackage{comment}

\definecolor{darkblue}{rgb}{0,0,0.35}
\definecolor{darkred}{rgb}{0.6,0,0}
\definecolor{darkgreen}{rgb}{0.1,0.35,0}
\definecolor{lightred}{rgb}{0.8,0.4,0.4}
\definecolor{purple}{rgb}{0.3, 0, 0.4}
\hypersetup{colorlinks, linkcolor=darkblue, citecolor=darkgreen, urlcolor=darkblue}

\newcommand*{\condition}[1]{Condition~\ref{#1}}

\newtheorem{theorem}{Theorem}[section]
\newtheorem{lemma}[theorem]{Lemma}

\newtheorem{proposition}[theorem]{Proposition}

\newtheorem*{limit-theorem}{Limit Theorem under \condition{eq:smooth}}
\theoremstyle{definition}
\newtheorem{definition}[theorem]{Definition}
\newtheorem{problem}[theorem]{Problem}

\newtheorem{observation}[theorem]{Observation}

\newtheorem{question}[theorem]{Question}

\DeclareMathOperator{\Unif}{Unif}

\renewcommand{\Pr}[1]{\mathbb{P}\left(#1\right)}
\newcommand{\Ex}[1]{\mathbb{E}\left[#1\right]}
\DeclareMathOperator{\Bin}{Bin}
\DeclareMathOperator{\Po}{Po}

\newcommand{\N}{\mathbb{N}}
\newcommand{\Rplus}{\mathbb{R}^+}
\newcommand{\R}{\mathbb{R}}

\newcommand{\diam}{\mathrm{diam}}

\newcommand{\eps}{\varepsilon}

\newcommand{\pwitv}[1]{\mathbf{#1}}

\newcommand{\G}{G_0}

\newcommand{\D}{\mathfrak D}

\newcommand{\jmax}{J}

\newcommand{\finseq}{\N^{< \omega}}

\def\an(#1){a(#1)}
\def\Q(#1){Q(#1)}
\def\ppR(#1){R(#1)}
\def\U(#1){U(#1)}

\newcommand{\nondecreasing}{non-decreasing}
\newcommand{\strictlyincreasing}{strictly increasing}

\begin{document}

\title{Explosion and linear transit times in infinite trees}
\author{Omid Amini} \address{\,CNRS - DMA, \'Ecole Normale Sup\'erieure, Paris, France.}  \email{oamini@math.ens.fr}
\author{Luc Devroye} \address{\,School of Computer Science, McGill University, Montreal, Canada.} \email{lucdevroye@gmail.com}
\author{Simon Griffiths} \address{\,Department of Statistics,
University of Oxford, Oxford, United Kingdom.}\email{simon.griffiths@stats.ox.ac.uk} 
\author{Neil Olver} 
\address{\,Department of Econometrics \& Operations Research, VU University Amsterdam, The Netherlands; and CWI, The Netherlands.}\email{n.olver@vu.nl}

\begin{abstract}
 Let $T$ be an infinite rooted tree with weights $w_e$ assigned to its edges. 
 Denote by $m_n(T)$ the minimum weight  of a path from the root to a node of the $n$th generation. 
 We consider the possible behaviour of $m_n(T)$ with focus on the two following cases: we say $T$ is explosive if 
\[
\lim_{n\to \infty}m_n(T)\, <\, \infty\,,
\]
and say that $T$ exhibits linear growth if 
\[
\liminf_{n\to \infty}\, \frac{m_n(T)}{n}\, > \, 0\,.
\]

We consider a class of infinite randomly weighted trees related to the Poisson-weighted infinite tree, 
and determine precisely which trees in this class have linear growth almost surely.  
We then apply this characterization 
to obtain new results concerning the event of explosion in infinite randomly weighted
spherically-symmetric trees, answering a question of 
Pemantle and Peres~\cite{PP}.  As a further application, we consider the random
real tree generated by attaching sticks of deterministic decreasing lengths, 
and determine for which sequences of lengths the tree has finite height almost surely. 
\end{abstract}

\maketitle

\section{Introduction}

Let i.i.d.\ random weights $w_e$ be assigned to the edges of an 
infinite rooted tree $T$, and let $m_n(T)$ 
denote the minimum weight of a path from the root to a node of the $n$th generation. 
In the context of first passage percolation, looking at the weight of an edge 
as the transition time between the two corresponding nodes, $m_n(T)$ is the first passage 
time to the $n$th generation.  
We consider the possible behaviour of $m_n(T)$ with particular focus on the following cases: 
we say $T$ is \emph{explosive} if 
\[
\lim_{n\to \infty}m_n(T)\, <\, \infty\,,
\]
and say that $T$ exhibits \emph{linear growth} if 
\[
\liminf_{n\to \infty}\, \frac{m_n(T)}{n}\, > \, 0\,.
\]

 In the case where the tree $T$ is itself a random Galton-Watson tree conditioned on the survival, 
 the quantity $m_n$ also occurs as the minimal $n$th generation position in a branching random walk in $\mathbb R$.  
 The linear growth property goes at least back to the work of Hammersley~\cite{Ham74}, 
 Kingman~\cite{Kin75}, and Biggins~\cite{Big76}. 
 Many other, including very recent, results~\cite{AdRe09, Ai11, AiSh11, ADGO, Bra78, DeHo91,  HuSh09, McD95, Vat96} on the behavior of $m_n$ in the context of branching random walks 
 are known, we refer to~\cite{ADGO} and the discussion in the 
 introduction there for a short survey. The literature on explosion is partially surveyed by Vatutin and
Zubkov~\cite{VZ93}.

 We assume now that the tree $T$ is deterministic. 
  Pemantle and Peres introduced the concept of stochastic dominance between trees and proved 
  in~\cite{PP} 
  that amongst trees with a given sequence of generation sizes, explosion is most likely in the case that 
  the tree is spherically symmetric. Recall that a tree $T$ is called spherically symmetric if all the vertices at generation $n$ have the same number $f(n)$ of 
  children, for some function $f: \mathbb N \cup \{0\} \rightarrow \mathbb N$. 
Pemantle and Peres proved that for a spherically symmetric tree $T$ with a non-decreasing 
branching function $f$, and with weights $w_e$ independent exponential random variables of mean one,  the probability of the event of 
explosion is 0 or 1 according to 
whether the sum $\sum_{n=0}^\infty f(n)^{-1}$ is infinite or finite. They also showed that the same statement holds for weight random variables with distribution function $G$ satisfying 
$\lim_{t\rightarrow 0} G(t)t^{-\alpha}=c >0$ for some $\alpha>0$. 
 Furthermore, they asked if the same simple explosion criterion holds for general 
edge weight distributions, under reasonable assumptions.

One of our aims in this paper is to answer, essentially completely, this question of Pemantle and Peres.
In order to do so, we consider a class of infinite weighted trees related to the Poisson-weighted infinite tree, the PWIT, 
introduced by Aldous~\cite{Ald92, AS}.  Since its introduction, the PWIT has been identified as the limit object of the solutions 
of various combinatorial optimization problems. The survey by Aldous and Steele~\cite{AS} provides a general overview with several examples of 
applications; see also~\cite{AdGK, BCC} for some more recent applications.

 As a cornerstone of all our results, we determine precisely which trees in this class of 
generalized PWITs have linear growth almost surely. We then present two
applications of this result:

\noindent $-$ First, we provide in Sections~\ref{sec:explosion} 
and~\ref{sec:examples} our results concerning the event of explosion in 
spherically-symmetric trees, generalizing the results of~\cite{PP}, and

\noindent $-$ Second, we consider general classes of 
random real trees constructed via a stick breaking process on $\Rplus$, 
in a similar way that Aldous' CRT~\cite{Ald91} is constructed, 
and give in Section~\ref{sec:stick} a criterion for these random real trees to have finite 
diameter almost surely.

In the remainder of this introduction, after summarizing our notation, we state our main results.

\subsection*{Basic definitions and notation}
 Given edge weights $w_e$ for each $e \in E(T)$, we write $w(\gamma)$ for the sum of the weights in a path $\gamma$.
    We write $T_n$ for the nodes of the $n$th generation (i.e., at distance $n$ from the root), and $\Gamma_{n}(T)$ for the family of paths from the root to $T_n$.  In this notation, $m_n(T):=\min_{\gamma\in \Gamma_n(T)}w(\gamma)$.

Recall that a spherically symmetric tree with \emph{branching function} $f:\N_0\to \N$ is 
the rooted tree $T_f$ 
in which the root has $f(0)$ children and each node of $(T_f)_n$ has $f(n)$ children (here 
$\mathbb N_0:=\mathbb N \cup\{0\}$). 
We write $F(n)$ for $\prod_{i=0}^{n-1}f(i)$, and note that $F(n)=|(T_f)_n|$.
We shall tend to focus on the case that $f$ is \nondecreasing{}. 

A Poisson point process of intensity $\lambda\in \Rplus$ is a point process $P$ on the positive real line such that for each pair of disjoint intervals $[a,b]$, $[c,d]$ we have
\begin{enumerate}[(i)]
\item $|P\cap [a,b]|$ is distributed as $\Po(\lambda(b-a))$, and,
\item $|P\cap [a,b]|$ and $|P\cap [c,d]|$ are independent.
\end{enumerate}

More generally, given a (measurable) function $\lambda:\Rplus \to \Rplus$, 
which is locally integrable and satisfies $\int_0^\infty \lambda = \infty$, the inhomogeneous Poisson point process $P^{\lambda}$ 
is a point-process on the positive real line such that for each pair of disjoint intervals $[a,b]$, $[c,d]$ we have
\begin{itemize}
\item[(i)] $|P^\lambda\cap [a,b]|$ is distributed as $\Po(\int_{a}^{b}\lambda)$, and,
\item[(ii)] $|P^\lambda\cap [a,b]|$ and $|P^\lambda\cap [c,d]|$ are independent.
\end{itemize}

We denote by $P(j)$ the position of the $j$th smallest particle of the point process $P$.

We now define a class of infinite trees that generalize the Poisson-weighted infinite tree (which corresponds to the case $\lambda$ is the constant function with value $1$).  
We shall use $\N^{< \omega}$ to denote the set of finite (ordered) 
sequences of natural numbers. A typical element of  $\N^{< \omega}$ is denoted by 
$\pwitv{i}$, and for an integer $j\in \N$, the sequence $\pwitv{i}j$ is obtained from $\pwitv{i}$ by inserting $j$ to the very right end of the sequence.

\begin{definition}
The P$^\lambda$WIT, which we denote by $T^{\lambda}$, has vertices labelled by $\finseq$, with $\emptyset$ labelling the root, and edge set 
\[
\{ \;\{\pwitv{i},\pwitv{i}j\}\;:\pwitv{i}\in \finseq, j\ge 1\}\, .
\]

Associate to each vertex $\pwitv{i}$ an independent point process $P^{\lambda}_{\textbf{i}}$ 
(distributed as $P^{\lambda}$), and give edge $\{\pwitv{i},\pwitv{i}j\}$ of $T^\lambda$ 
the weight $P^{\lambda}_{\textbf{i}}(j)$.
\end{definition}

\subsection{Linear growth in generalizations of the PWIT}
The Poisson-weighted infinite tree exhibits linear growth almost surely.  It is therefore natural to ask how general this property is in the generalizations of the PWIT defined above.  We answer this question completely.  Furthermore, we provide exponential probability bounds for the event that $m_n(T)$ grows more slowly.

\begin{theorem}\label{thm:PlambdaWITbis}
Let $\lambda: \Rplus \to \Rplus$ be any locally integrable function with $\int_0^\infty \lambda = \infty$. 
Then we have the following dichotomy.

\begin{enumerate}[(i)]
    \item If either there exists some $t>0$ such that 
    $\int_0^t\lambda =0$, or there exists some $C > 0$ so that $\int_0^x \lambda \leq C^x$ 
    for all $x \in \Rplus$, then there exists $\alpha>0$ such that
\[
\lim_{n\rightarrow \infty} \frac {m_n(T^{\lambda})}{n} = \alpha
\]
almost surely. In particular, $T^\lambda$ has linear growth, almost surely.

Furthermore, for each $K>0$, there exists $\delta>0$ such that
\[
\Pr{m_n(T^{\lambda})<\delta n}\, \le \, e^{-Kn}\, .
\]

\item Otherwise, $T^\lambda$ does not have linear growth, almost surely.
\end{enumerate}
\end{theorem}

Note that if $\int_0^t \lambda=0$ for some $t>0$, then obviously $T^\lambda$ has linear growth. 
The main part of the theorem thus concerns the existence or the non-existence of 
a constant $C>0$ so that 
$\int_0^x \lambda \leq C^x$ for all $x\in \mathbb R^+$, which provides a dichotomy for 
linear growth 
in the trees $T^{\lambda}$. The proof appears in Section~\ref{sec:PlambdaWIT}. 

\subsection{Explosion in infinite trees}

Recall that we call a rooted weighted tree \emph{explosive} if 
\[
\lim_{n\to \infty}m_n(T)\, <\, \infty\, .
\]
Pemantle and Peres proved that amongst trees with a given sequence of generation sizes, 
explosion is most likely in the case that the tree is spherically symmetric.

Let $f:\N_0 \to \N$, and let $T_f$ denote the spherically-symmetric tree in which each node $v$ of generation $n$ has $f(n)$ children.  Given a distribution function $G:[0,\infty)\to [0,1]$, let $T_f^G$ denote the randomly weighted tree obtained by giving each edge of $T_f$ an i.i.d.\ weight distributed  according to $G$.

\begin{definition} Given the distribution $G$ and \nondecreasing{} 
function $f:\mathbb{N}_0\to \mathbb{N}$, we say that $f$ is \emph{$G$-explosive} 
if $T^G_f$ is explosive almost surely.  
\end{definition}

It is easy to see that a sufficient condition for $f$ being $G$-explosive 
is (a slightly stronger version of) the ``local min-summability'' condition
(compare to the ``global min-summability'' condition considered in~\cite{ADGO}): apply a greedy algorithm to 
construct an infinite path in the tree $T_f^G$ by starting from the root, and by choosing recursively for the end vertex 
$v_n$ of the already constructed path up to level $n$, the minimum weight edge $v_nv_{n+1}$ among the $f(n)$ adjacent 
edges to level $n+1$, see Proposition~\ref{prop:onepluseps}. This motivates the following definition.

\begin{definition} Given the distribution $G$ and a \nondecreasing{} function $f:\mathbb{N}\to \mathbb{N}$, we say that $f$ is \emph{$G$-small} if
\begin{equation}\label{Gsmall}
\sum_{n\ge 0} G^{-1}(f(n)^{-1})\, < \, \infty\, .
\end{equation}
\end{definition}

One may now interpret Pemantle and Peres~\cite[Page 193]{PP} as asking the following.

\begin{question}[Pemantle-Peres~\cite{PP}]\label{question:PP} For which $G$ does the equivalence
\[
f \text{ is  $G$-small} \, \Leftrightarrow \, f \text{ is $G$-explosive}\,
\]
hold in the class of \nondecreasing{} functions $f:\N_0 \to \N$?
\end{question}
Pemantle and Peres~\cite{PP} 
showed that if $G$ has a limit law at $0$ in the sense that 
$\lim_{x\rightarrow 0}G(x)x^{-\alpha}$ exists and is positive, then this equivalence holds.
They speculated that perhaps the equivalence holds provided $G$ is continuous and \strictlyincreasing{}.  
The following definition encodes robust versions of the properties of being continuous and 
\strictlyincreasing{}.  
\begin{definition} The distribution $G$ is \emph{controlled near $0$} if
\[
1\, <\, \liminf_{x\to 0}\frac{G(cx)}{G(x)}\, \leq \, \limsup_{x\to 0}\frac{G(cx)}{G(x)}\, <\, \infty\, ,
\]
for some constant $c>1$.
\end{definition}
Note that this is much weaker than the requirement that 
$G$ be regularly varying around $0$, which corresponds to the condition that the 
limit exists for any $c>1$ and lies in $(1,\infty)$, 
which is in turn weaker than the requirement that $G$ has a limit law at $0$. Thus, the following generalizes the result of 
Pemantle-Peres~\cite{PP} and, in spirit, confirms the validity of their speculation.

\begin{theorem}\label{thm:main} If $G$ is controlled near $0$, then the equivalence
\[
f \text{ is $G$-small} \, \Leftrightarrow \, f \text{ is $G$-explosive}\,
\]
holds in the class of \nondecreasing{} functions $f:\N \to \N$.
\end{theorem}

On the other hand, we give examples that demonstrate that the ``controlled near $0$'' condition cannot be significantly weakened. 
Firstly, we show that the equivalence  may fail in general in both 
directions even if we assume $G$ continuous and \strictlyincreasing{}:

\begin{proposition}\label{prop:examples} There exist a continuous, 
\strictlyincreasing{} weight distribution $G$, and \nondecreasing{} functions 
$f_1,f_2:\N_0\to \N$ with the following properties. 
\begin{enumerate}[(i)]
\item\label{itm:simplenoexpl} The function $f_1$ is $G$-small but not $G$-explosive.
\item\label{itm:explnotsmall} The function $f_2$ is $G$-explosive but not $G$-small.
\end{enumerate}
\end{proposition}

Secondly, a counterexample of the form (\ref{itm:explnotsmall}) holds in fact for a rather large class of weight distributions: 
\begin{theorem}\label{thm:lotsofcounters}
Let $G$ be any weight distribution satisfying either
\begin{equation}\label{cond1}
\limsup_{i\to \infty} \frac{G(x_i)}{G(x_i/c)}\, <\, \limsup_{i\to \infty} \frac{G(cx_i)}{G(x_i)}\, =\,  \infty\, ,
\end{equation}
or
\begin{equation}\label{cond2}
1 \, = \, \liminf_{i\to \infty} \frac{G(x_i)}{G(x_i/c)}\, < \, \liminf_{i\to \infty} \frac{G(cx_i)}{G(x_i)}\, ,
\end{equation}
for some constant $c>0$ and decreasing sequence $x_i:i\ge 1$ with limit $0$.  
Then there exists  a function $f: \N_0 \to \N$ which is $G$-explosive but not $G$-small.
\end{theorem}
See Section~\ref{sec:examples} for more details.

\medskip

 While Theorem~\ref{thm:main} and Theorem~\ref{thm:lotsofcounters} together cover most naturally defined distribution functions $G$, we would like to stress that it is still open to answer Question~\ref{question:PP} completely.

\subsection{Finite height criterion for stick breaking random real trees}

Consider the following method for constructing a random real tree.  
Given a sequence $\ell(i):i\in \mathbb{N}$, define the real tree $A_\ell$ 
recursively as follows.  
Let $A_{\ell}(1)$ consist of a closed segment of length $\ell(1)$ rooted at one end, and 
for each $i\ge 1$, define $A_{\ell}(i+1)$ by attaching one end of a closed segment
of length $\ell(i+1)$ to a uniformly randomly chosen point of the tree $A_{\ell}(i)$. Let
\[
A_{\ell}^o\, :=\, \bigcup_{i\ge 1}A_{\ell}(i)\, 
\]
and define $A_\ell$ as the completion of $A_\ell^o$. The random real tree $A_\ell$ is referred to as the random real tree given by the stick 
breaking process obtained by cutting the positive real line according to the segment lengths sequence 
$\ell$.

Note that in the case where the sequence $\ell(i)$ is the length of the segment $[P^\lambda(i), 
P^\lambda(i+1)]$ given by an inhomogeneous Poisson point process $P^\lambda$ with intensity 
$\lambda(t)=t$ on $\Rplus$, the random real tree $A_\ell$ is precisely the 
continuum random tree constructed by Aldous~\cite{Ald91}. 
The inhomogeneous, and more general, versions of this construction are treated 
in~\cite{Ald93, AlPi00, GH14}.  

Curien and Haas~\cite{CH} have recently studied
the geometric properties of such trees (such as compactness and Hausdorff dimension) 
in the case of deterministic lengths $\ell(i)$ which decay 
roughly like a power  $\ell(i) \approx i^{-\alpha}$ for $\alpha >0$. 
It was a question of Curien~\cite{Curien}  
that led us to consider Problem~\ref{qu1} below.

For a real tree $A$, we denote by $d(A)$ the height of $A$, i.e., the supremum of distances from  
points of $A$ to the root. We denote by $\diam(A)$ the diameter of $A$. 
Note that $d(A)\leq \diam(A) \leq 2 d(A)$. 

The following is then a very natural problem: 
\begin{problem}\label{qu1}
Classify all sequences $\ell(i), i\in \mathbb{N}$, for which
we have $d(A_{\ell})<\infty$, or equivalently, $\diam(A_\ell) < \infty$, almost surely.
\end{problem}
Note that the property of having 
bounded diameter almost surely is equivalent to the almost sure 
compactness of the random real tree $A_\ell$~\cite{CH}.

As an application of our result on linear growth of P$^\lambda$WIT, we answer this question completely in Section~\ref{sec:stick} 
for those length sequences $\ell(i)$ which are deterministic and decreasing. 

\begin{theorem}\label{thm:stick}
Let $\ell(i), i\in \mathbb{N}$, be a decreasing sequence.  
Then $d(A_{\ell})<\infty$ almost surely if and only if 
$\,\,\sum_{n\geq 1}\frac {\ell(n)}n \,<\,\infty\,, \,\textrm{or equivalently,}\,$ 
if and only if $\,\,\sum_{n\ge 1}\ell(2^n)\, <\, \infty\, .$
\end{theorem}
 
We note that the requirement that 
$\ell(i)$ be decreasing may be relaxed: writing $\bar{\ell}(i)$ for the average value of 
$\ell(j):j\le i$, we only need to assume that the ratio 
  $\ell(i)/\bar{\ell}(i)$ is bounded. 
  
   It might be possible that a  similar criterion (applied to the decreasing rearrangement of $\ell$) remains valid for a general sequence 
  $\ell(i)$, a question we leave open. 
 
 Finally, we mention that Curien and Haas  have independently and simultaneously 
 proved in~\cite{CH}, by using different tools, that 
  if $\ell(i) \leq i^{-\alpha+o(1)}$ for some $\alpha>0$, 
  then the random real tree $A_\ell$ has bounded height almost surely. 
  They further prove (amongst other things) that under the  
  additional assumption on the average $\bar\ell(i) = i^{-\alpha+0(1)}$ for $\alpha\in(0,1]$,  
  the Hausdorff dimension of $A_\ell$ is $\alpha^{-1}$, 
  while for $\alpha>1$, the Hausdorff dimension is one, almost surely.

\section{Linear growth in P$^{\lambda}$WIT trees}\label{sec:PlambdaWIT}
In this section, we will prove Theorem~\ref{thm:PlambdaWITbis}.

Hammersley~\cite{Ham74}, Kingman~\cite{Kin75}, and Biggins~\cite{Big76} give 
general conditions under which linear growth occurs in a branching random walk, as well as 
associated limit theorems. The version most suitable to our situation is Kingman~\cite{Kin75} 
(see also Biggins~\cite{Big77}). Kingman's theorem holds for general point processes on $\Rplus$, 
but we state it here only for the case relevant to us, that of inhomogeneous Poisson point 
processes.
\begin{theorem}[Kingman~\cite{Kin75}, specialized to the P$^\lambda$WIT]\label{thm:genlingrowth}     Let $T^\lambda$ be a P$^\lambda$WIT, for some appropriate $\lambda: \Rplus \to \Rplus$.    Define $\mu: \Rplus \to \Rplus \cup \{\infty\}$ by
    \begin{equation}\label{eq:mu} \mu(a) := \inf_{D \geq 0} \left\{ e^{D a} \Ex{\sum_{j \geq 1} e^{-D P^\lambda(j)}}\right\}. 
    \end{equation}
    Assume $\mu(a) < \infty$ for some $a > 0$.
    Then 
    \[ \lim_{n \to \infty} m_n(T^\lambda) / n \to \alpha \,\,\,\text{almost surely,}  \]
    where $\alpha = \inf \{ a : \mu(a) > 1 \}$.
    In particular, if $\alpha > 0$, then $T^\lambda$ exhibits linear growth almost surely.
    Moreover, 
    \[ \Pr{m_n(T^\lambda) \leq an} \leq \mu(a)^n. \]
\end{theorem}

We begin with part (i) of Theorem~\ref{thm:PlambdaWITbis}. 
The following argument allows to treat only the case where 
$\int_0^t \lambda>0$ for any $t>0$. Indeed, 
if $\int_0^c \lambda = 0$ for some $c > 0$, then $\Pr{m_n(T^\lambda) < c n} = 0$ and 
linear growth holds trivially. In addition, if $c$ is chosen maximum with respect to 
$\int_0^c\lambda=0$, then 
the intensity function $\tilde \lambda$ defined by $\tilde \lambda(t) := \lambda(t+c)$ 
for any $t\in \mathbb R^+$ verifies $\int_0^t\tilde \lambda >0$ 
for all $t>0$, so that the limit assertion for $\lambda$ follows from that of $\tilde\lambda$, 
guaranteed either by part (i) or part (ii) of the theorem.

So suppose that $\int_0^t \lambda>0$ for any $t>0$. Let $C>0$ so that $\int_0^x \lambda \leq C^x$ for all $x \geq 1$.
The key lemma is the following, after which we will be able to directly apply Kingman's result. \begin{lemma}\label{lem:onelevel} 
For every $\eta>0$, there exists $D\in \R$ such that
\[
\Ex{\sum_{j\ge 1}e^{-DP^{\lambda}(j)}}\, <\, \eta\, .
\]
\end{lemma}

\begin{proof} Note that, by a straightforward coupling, if $\lambda_1$ and $\lambda_2$ are such that $\int_0^x \lambda_1\ge \int_0^x \lambda_2$ for all $x\ge 0$, then
\[
\Ex{\sum_{j\ge 1}e^{-DP^{\lambda_1}(j)}}\, \le \, \Ex{\sum_{j\ge 1}e^{-DP^{\lambda_2}(j)}}\, .
\]
This allows us to assume $\lambda$ is decreasing on $[0,1]$, and that $\lambda(y)=(\log{C})C^y$ for $y\ge 1$.  
In particular, if we prove the lemma for such functions $\lambda$ then this proves the lemma in general. 
In addition, we may obviously 
assume that $\eta \leq 1$.

We shall deal separately with points in the interval $[0,1]$ and those in $(1,\infty)$.  We give a choice of $D$ such that
\[
\Ex{\sum_{j\ge 1:P^{\lambda}(j)\le 1}e^{-DP^{\lambda}(j)}}\, \le \, \frac{\eta}{2}\qquad \text{and}\qquad \Ex{\sum_{j\ge 1:P^{\lambda}(j)> 1}e^{-DP^{\lambda}(j)}}\, \le \, \frac{\eta}{2} \, .
\]
Let $D_1$ be large enough that
\[
\Ex{e^{-D_1P^{\lambda}(1)}}\, \le \frac{\eta}{4}\, .
\]
For each $j\ge 2$, the fact that $\lambda$ is decreasing on $[0,1]$ implies that conditioned on $P^{\lambda}(j)\leq 1$, the 
increment $P^{\lambda}(j)-P^{\lambda}(j-1)$ stochastically dominates $P^{\lambda}(1)$, and so
\[
\Ex{e^{-D_1P^{\lambda}(j)}}\, \le \, 
\left(\Ex{e^{-D_1P^{\lambda}(1)}}\right)^j \, \le \,
\left(\frac{\eta}{4}\right)^j\, .
\]
Summing the above inequalities, we obtain that for any $D\ge D_1$,
\[
\Ex{\sum_{j\ge 1:P^{\lambda}(j)\le 1}e^{-D P^{\lambda}(j)}}\, \le \, \frac{\eta}{2}\, .
\]
Now, for the points in $(1,\infty)$, we have
\[
\Ex{\sum_{j\ge 1:P^{\lambda}(j)> 1}e^{-DP^{\lambda}(j)}}\, \le \, 
\int _{1}^{\infty} (\log{C})C^y e^{-Dy}\, dy\, \le \frac{C\log{C}}{D-\log{C}}\,\,e^{-D}\,  .
\]
Therefore, taking $D\ge D_1$ large enough so that $C\log{C}e^{-D}/(D-\log{C})\le \eta/2$, we have
\[
\Ex{\sum_{j\ge 1:P^{\lambda}(j)> 1}e^{-DP^{\lambda}(j)}}\, \le \, \eta/2\, ,
\]
which completes the proof of the lemma.
\end{proof}

Note that the lemma ensures that $\mu(a)<\infty$ for any $a \in \Rplus$.
Thus by Theorem~\ref{thm:genlingrowth}, $m_n(T^\lambda) / n \to \alpha$ as $n \to \infty$, where $\alpha := \inf\{a : \mu(a) > 1\}$. 
The lemma also immediately implies that $\mu(a) \to 0$ as $a \to 0$.
So $\alpha > 0$, and thus $T^\lambda$ exhibits linear growth almost surely.
Moreover, again by Theorem~\ref{thm:genlingrowth}, 
\[ \Pr{m_n(T^\lambda) \leq \delta n} \leq \mu(\delta)^n \leq e^{-Kn} \]
for any $\delta>0$ s.t.\ $\mu(\delta) \leq e^{-K}$.
This completes the proof of part~(i) of Theorem~\ref{thm:PlambdaWITbis}.

\medskip

We now move to part (ii).  So suppose that for any $t>0$, we have $\int_0^t \lambda >0$, and 
$\lambda$ does not satisfy the conditions of part~(i): in other words, 
for any $C > 0$, there exists an $x \geq 1$ so that $\int_0^x \lambda > C^x$. It is not 
possible to directly apply the result of Kingman, since it turns out that $\mu(a) = \infty$ for all $a$.
A truncation argument may be used, but we prefer to give here a direct proof.

We will prove that $T^{\lambda}$ a.s. does not have linear growth by finding a value $n_\delta$ for every $\delta > 0$, 
with $n_\delta \to \infty$ as $\delta \to 0$, such that
\[
\Pr{m_{n_\delta}(T^{\lambda}) \geq \delta \cdot n_\delta}\, \to \, 0 \quad \text{as} \quad \delta \to 0.
\]

So fix any $\delta > 0$.
Let $p=p_\delta := \min \{\Pr{P^\lambda(1) < \delta/2}, \delta/2\}$. 
Note that $p > 0$ by the assumption that $\int_0^t \lambda > 0$ for all $t > 0$. 
Also let $C =C_\delta := p^{-8p^{-1}}$, and choose $x =x_\delta \geq 1$ so that 
$\int_0^x \lambda > C^x$. Let
$n = n_\delta := \lceil 2x / \delta \rceil$. Note that $n_\delta 
\to \infty$ as $\delta \to 0$.

Let $\mathfrak T $ denote the connected component containing the root of the subforest of 
$T^{\lambda}$ obtained by keeping those edges adjacent to the root with weight less than $x$ and all remaining edges having weight less than $\delta/2$.
We thus have 
\begin{align*}
    \Pr{m_n(T^{\lambda}) \geq \delta n)} &\leq \Pr{\mathfrak T_n = \emptyset}\\
    &\leq \Pr{\mathfrak T_n = \emptyset \mid\,|\mathfrak T_1| \geq p^{-2n}} + \Pr{|\mathfrak T_1| < p^{-2n}}.
\end{align*}
Given that a particular child $v$ of the root is in $\mathfrak T$, the probability that $v$ has a descendant
in $\mathfrak{T}_n$ (generation $n$) is certainly at least $p^{n-1}$. Therefore, 
$$\Pr{\mathfrak T_n = \emptyset \mid\,|\mathfrak T_1| 
\geq p^{-2n}} \leq \Pr{\Bin(\lceil p^{-2n} \rceil, p^{n-1}) = 0}.$$
Also $|\mathfrak T_1|$ is Poisson distributed, with mean 
\[ \int_0^{x} \lambda \geq C^x = p^{-8p^{-1}x} \geq p^{-2p^{-1}n\delta} \geq p^{-4n}. \]
So clearly the above bound on $\Pr{m_n(T^\lambda \geq \delta n)}$ goes to zero as $\delta \to 0$, and the result is proved.

\section{Explosion in spherically symmetric trees}\label{sec:explosion}

In this section we prove Theorem~\ref{thm:main}.  That is, 
we show that if $G$ is controlled near $0$, and $f:\N_0\to \N$ is \nondecreasing{}, then 
\[
f \text{ is $G$-small}\, \Leftrightarrow \, f \text{ is $G$-explosive}\, .
\]
One direction of the equivalence is straightforward.  If $G$ is controlled near $0$ and $f$ is $G$-small,
then one easily deduces that
\[
\sum_{n\ge 1}G^{-1}\left(\frac{1+\eps}{f(n)}\right)\, <\, \infty \, ,
\]
for some $\eps>0$.  
The first direction of the equivalence then follows from the following proposition.

\begin{proposition}\label{prop:onepluseps}
Let $G$ be an arbitrary weight distribution function, and let $f:\N_0\to \N$ be \nondecreasing{}.  If
\[
\sum_{n\ge 1}G^{-1}\left(\frac{1+\eps}{f(n)}\right)\, <\, \infty \, ,
\]
then $f$ is $G$-explosive.
\end{proposition}

\begin{proof}  Note that the summability condition certainly implies that $f(n)\to \infty$. 
Consider the forest $F$ in which edges from generation $n$ to generation $n+1$ are kept if 
their weight is at most $G^{-1}((1+\eps)/f(n))$.  
Since any path in $F$ has finite weight it suffices to show it contains an 
infinite path with positive probability.  
Each node of generation $n$ has $\Bin(f(n),(1+\eps)/f(n))$ children, 
which stochastically dominates the distribution $\max\{\Po(1+\eps/2),\eps^{-1}\}$ when $n$ is large.  
Since the Galton-Watson branching process with offspring distribution 
$\max\{\Po(1+\eps/2),\eps^{-1}\}$ survives with positive probability, the same is true for $F$.
\end{proof}

We now turn to the remaining direction of the equivalence.  We prove that if $G$ is controlled near $0$ and $f$ is not $G$-small, i.e., 
\[
\sum_{n\ge 1}G^{-1}(f(n)^{-1})\, =\, \infty\, ,
\]
then $f$ is not $G$-explosive.

The idea is to compare weights along paths in $T_f^G$ with the terms of the sequence $a(n):=G^{-1}(f(n)^{-1})$.  Indeed, the key intermediate result will be Proposition~\ref{hatlinear}, which claims that this re-normalized weighted tree has linear growth (at least after the removal of some extra heavy edges).

\begin{definition}
Given the weighted infinite tree $T_f^G$, and a sequence $(a(n))$, 
define the renormalized weighted tree $\hat{T}_f^G$ to have the same underlying graph as $T_f^G$, but with weights
\[
\hat{w}_e\, :=\, \frac{w_e}{a(|e|)}\qquad e\in E(\hat{T}_f^G) = E(T_f^G)\, ,
\]
where $w_e$ is the weight of $e$ in $T_f^G$ and $|e|$ is the generation of the parent in the edge $e$.
\end{definition}

Say that $\hat{T}_f^G$ has been $\eta$-trimmed, if all edges $e$ with
\[
\hat{w}_e\, \ge\, \frac{\eta}{a(|e|)}\, 
\]
have been removed.  Note, this is equivalent to removing all edges of weight at least $\eta$ in $T_f^G$ before renormalizing. 

\begin{proposition}\label{hatlinear}
Let $G$ be controlled near $0$ and let $f$ be a \nondecreasing{} function with $f(n)\to \infty$ as $n\to \infty$.  Then there exists $\eta>0$ such that the tree $\hat{T}$ obtained when $\hat{T}_f^G$ is $\eta$-trimmed exhibits linear-growth almost surely.
\end{proposition}
Let us show how to complete the proof of Theorem~\ref{thm:main}, by using 
Proposition~\ref{hatlinear}. The following lemma is essentially what we need.
\begin{lemma}\label{lem:infinite}
 Let $\epsilon>0$ and consider a sequence of non-negative real numbers $b(i):i\in \mathbb N$ such that
 $\sum_{i=1}^n \frac{b(i)}{a(i)} \geq \epsilon n$ for all sufficiently large $n$.  Then $\sum_{i=1}^{\infty} b(i) = \infty$
\end{lemma}

\newcommand{\Ss}[1]{S(#1)}
\begin{proof}
 Let $n_0$ be such that the claimed inequality holds for all $n\ge n_0$.  
 We prove that for any $n$, $\sum_{i=1}^n b(i) \geq \epsilon \sum_{j=n_0}^n a(j)$, which proves the lemma, since the later sum is assumed to be divergent.  Let $c(i):=\frac {b(i)}{a(i)}$ and $\Ss j := \sum_{i=1}^j c(i)$.  By assumption, for any $j\geq n_0$, we have $\Ss j \geq \epsilon j$.
 We now have that
 \begin{align*}
  \sum_{i=1}^n b(i)&= \sum_{i=1}^n c(i)a(i) = \sum_{i=1}^n c(i)\bigl[a(n)+\sum_{j=i}^{n-1}(a(j)-a(j+1))\bigr]\\
  &= \Ss n a(n) + \sum_{j=1}^{n-1} \Ss j (a(j)-a(j+1)) \\
  &\geq \Ss n a(n) +  \sum_{j=n_0}^n \Ss j (a(j)-a(j+1))\\
& \geq \epsilon n a(n) + \sum_{j=n_0}^{n-1} \epsilon j (a(j)-a(j+1))\geq \epsilon \sum_{j=n_0}^{n} a(j).
\end{align*}
The penultimate inequality uses that $a(n)$ is a decreasing sequence. 
Since $\sum_{j=n_0}^\infty a(j) = \infty$, the lemma follows.
\end{proof}

\begin{proof}[Proof of Theorem~\ref{thm:main}]
Let $G$ be controlled near $0$ and $f:\N_0\to N$ a \nondecreasing{} function.  The proof that $f$ being $G$-small implies $f$ is $G$-explosive follows directly from Proposition~\ref{prop:onepluseps}, as explained above.

For the other direction, suppose that $f$ is not $G$-small, i.e.,
\[
\sum_{n\ge 1}G^{-1}(f(n)^{-1})\, =\, \infty\, ,
\]
we shall prove that $f$ is not $G$-explosive.  

In the case that $f$ is bounded the proof is easy.  Let $D$ be the maximum value of $f(n)$, and let $\delta_0>0$ be such that $G(\delta_0)<1/2D$.  All components of light edges (i.e., with weights at most $\delta_0$) are finite almost surely (since they correspond to sub-critical branching processes).  So every infinite path contains infinitely many edges of weight at least $\delta_0$, which demonstrates that $f$ is not $G$-explosive.

If $f(n)$ is unbounded, we shall use Proposition~\ref{hatlinear}.  First observe that we may remove all edges of weight above some $\eta>0$.  Indeed, let $E$ be the event $T_f^G$ contains an infinite path of finite weight and let $E(\eta)$ be the event $T_f^G$ contains such a path with all edge weights at most $\eta$.
It is elementary that the ratio of $\Pr{E}$ and $\Pr{E(\eta)}$ is a constant.  So to prove that $\Pr{E}=0$, it suffices to prove that $\Pr{E(\eta)}=0$.

The renormalized version of the remaining tree is precisely the tree $\hat{T}$ obtained after $\hat{T}_f^G$ is $\eta$-trimmed.  By Proposition~\ref{hatlinear}, there exists (almost surely), a constant $\delta>0$ such that $m_n(\hat{T})\ge \delta n$ for all sufficiently large $n$.  Now, let $\gamma$ be any infinite path descending from the root in the tree $T_f^G$ and using only edges with weight at most $\eta$.  Writing $b(n)$ for the weights along this path, $a(n)$ for $G^{-1}(f(n)^{-1})$ and $c(n)$ for the ratio $b(n)/a(n)$, we have that $c(n)$ are exactly the weights on the corresponding path in $\hat{T}$ and so satisfy $\sum_{i=1}^{n}c(i)\ge \delta n$ for all sufficiently large $n$.  We are now in the setting of Lemma~\ref{lem:infinite}, and we deduce that $\sum_{n\ge 0}b(n)$ is divergent.

Since the choice of the path $\gamma$ was arbitrary it follows that (almost surely) $T_f^G$ does not contain an explosive path with all weights at most $\eta$, as required.
\end{proof}

In the rest of this section, we provide the proof of Proposition~\ref{hatlinear}.
\subsection{Proof of~Proposition~\ref{hatlinear}}
It will be convenient to relabel the vertices of $\hat{T}_f^G$ (and so also $\hat{T}$) by sequences in $\finseq$.
We label the root by $\emptyset$, and for any vertex $\pwitv{i} \in \finseq$ which labels a vertex in $\hat{T}_f^G$, we label with $\pwitv{i}j$ the child of $\pwitv{i}$ for which $\hat{w}_{\pwitv{i},\pwitv{i}j}$ has the $j$'th smallest value, amongst all weights of edges to children of $\pwitv{i}$.

With Theorem~\ref{thm:PlambdaWITbis} in mind, it will suffice to couple the tree $\hat{T} $ (obtained after $\hat{T}_f^G$ is $\eta$-trimmed) with a P$^{\lambda}$WIT (with $\lambda$ controlled by an exponential) in such a way that weights in $\hat{T}$ are at least the equivalent weights in the P$^{\lambda}$WIT.  
Such a coupling is provided by the following lemma.

\begin{lemma}\label{lem:lambda} Let $G$ be controlled near $0$, and let 
$f:\N_0\to \N$ be a \nondecreasing{} function with $f(n)\to \infty$ as $n\to \infty$.  
Then there exist constants $n_0\in \N$ and $\eta,C>0$, a function $\lambda:\Rplus\to\Rplus$ with $\int_0^x\lambda\le C^x$ for all $x\in \Rplus$, and a coupling between $\hat{T}$ (obtained from $\eta$-trimming $\hat{T}_f^G$) and a P$^{\lambda}WIT$ $T^\lambda$ in which
\[
    \hat{w}_{\pwitv{i},\pwitv{i}j} \geq P^{\lambda}_{\pwitv{i}}(j)
\]
for any edge $\{\pwitv{i},\pwitv{i}j\}$ in $\hat{T}$ with $\pwitv{i}$ at level at least $n_0$.
\end{lemma}

\begin{proof}
    We defer for now the explicit definitions of $n_0$, $C$, $\eta$ and $\lambda$.
    It is sufficient to show that for any $n \geq n_0$, and any vertex $\pwitv{i} \in \hat{T}_n$, we can couple the weights of the downward edges from $\pwitv{i}$ with the inhomogeneous Poisson point process $P^\lambda$.     The lemma then follows by applying the coupling at every vertex at distance $n_0$ or more from the root.

The weights $\hat{w}_{\pwitv{i}, \pwitv{i}j}:j=1,\dots ,f(n)$ are obtained by taking i.i.d.\ samples from the distribution $G$, dividing by $a(n)$, and then arranging in an increasing order. 
Equivalently, they are generated as
\[
\hat{w}_{\pwitv{i}, \pwitv{i}j}=\frac{G^{-1}(U(j))}{a(n)}
\]
where $U(1),\dots ,U(f(n))$ are i.i.d.\ $\Unif(0,1)$ random variables arranged in increasing order.  

Let $\eta_0=G^{-1}(1/2)$. We will later choose $\eta \leq \eta_0$,  thus ensuring $\eta$ to
satisfy $G(\eta) \leq 1/2$.

We can ignore all edges $\{\pwitv{i},\pwitv{i}j\}$ for which $U(j)\ge G(\eta)$, since these will disappear when we form $\hat{T}$ 
by $\eta$-trimming $\hat{T}_f^G$. 
Let $\jmax$ denote the largest index $j \leq f(n)$ s.t. $U(j) < G(\eta)$.

Thus one may couple the uniforms $U(1),\dots ,U(f(n))$ with a Poisson point process $Q$ of intensity $2f(n)$ on the positive real line so that 
\[
U(j)\, \ge \, Q(j) \land G(\eta) \qquad \text{for all } j\ge 1\, .
\]
It follows that there is a coupling in which
\begin{equation}\label{coup1}
\hat{w}_{\pwitv{i}, \pwitv{i}j}\, \ge \, \frac{G^{-1}(Q(j))}{a(n)} \qquad \text{for all } j=1,\dots , \jmax\, .
\end{equation}
So to prove the lemma, it suffices to demonstrate a coupling between $Q$ and $P^\lambda$ 
(for our choice of $\lambda$, which will be given shortly) such that
\[
\frac{G^{-1}(Q(j))}{a(n)}\, \ge \, P^{\lambda}(j) \qquad \text{for all } j=1,\dots, \jmax.
\]
The remainder is concerned with using the properties of $G$ to prove this statement.  

Since $G$ is controlled near $0$, there exist constants $c>1$ and $K\ge \max\{1/\eta_0,2\}$ 
    such that
\[
1+\frac{1}{K}\, <\, \frac{G(cx)}{G(x)} <\, K\, , \qquad{\text{for all }}x\le 1/K.
\]
Let $\eta = 1/K$ (note that then $\eta \leq G^{-1}(1/2)$, as assumed earlier).
It follows easily from the above, combined with the monotonicity of $G$, that for any $x_0 \in (0, \eta]$, 
\begin{align*} 
    G(x) &\leq K\cdot (x/x_0)^{\log (1 + 1/K)/\log c} G(x_0) \quad \forall x \in (0, x_0],\\
    G(x) &\leq K \cdot (x/x_0)^{\log K / \log c} G(x_0) \quad \forall x \in [x_0, \eta].
\end{align*}
More succinctly, we have that 
\begin{equation}\label{eq:Gcontrol}
G(x) \leq h(x/x_0) G(x_0) \qquad \text{for all } x, x_0 \in (0, \eta], 
\end{equation}
where
\begin{equation}\label{eq:h}
    h(z) = \begin{cases} K z^{\log(1+1/K)/\log c} & z \leq 1\\
        K z^{\log K / \log c} & z > 1 
    \end{cases}. 
    \end{equation}
One can think of this as providing some uniform control over the shape of $G$ when ``zooming in'' around some point $x_0$.
We now state our definition of $n_0$: it is chosen such that $f(n_0)\ge  K$ and $a(n_0) \leq \eta$. Let us see what the above tells us about the behaviour of $G$ about $a(n)$, for $n \geq n_0$.  
It yields
\begin{equation}\label{eq:Gcontroln}
G(x) \leq h(x/a(n)) G(a(n))\, =\, \frac{h(x/a(n))}{f(n)} \qquad \text{for all } x\in (0, \eta].
\end{equation}
Making the substitution $y=G(x)$ and applying $h^{-1}$, we obtain
\[
h^{-1}(yf(n)) \, \le\, \frac{G^{-1}(y)}{ a(n)}\,  \qquad \text{for all } y \in (0, G(\eta)].
\]
Since $Q(j) \leq G(\eta)$ for all $j \leq \jmax$, 
\begin{equation}\label{coup2}
 h^{-1}\bigl(Q(j)f(n)\bigr) \leq \frac{G^{-1}\bigl(Q(j)\bigr)}{a(n)} \qquad \text{ for all } j=1, \ldots, \jmax\, .
\end{equation}

We are now ready to give our choice of $\lambda$, and finish the proof of the lemma: 
simply define $\lambda := 2\frac{dh}{dx}$. 
Note that by the explicit form of $h$ given in~\eqref{eq:h}, it is clear that $h$ grows polynomially and so
there clearly exists a constant $C>0$ so that $\int_0^x \lambda = 2h(x) \leq C^x$ for all $x > 0$.
Moreover, the Poisson point process $P^\lambda$ with intensity $\lambda$ can be generated 
by $P^\lambda(j) = h^{-1}(Q(j)f(n))$, because $Q(j)f(n)$ is a Poisson point process of intensity $2$; 
this combined with~\eqref{coup1} and~\eqref{coup2} provide the required coupling.
\end{proof}

Take $n_0, \eta, C$ and $\lambda$ as guaranteed by the above lemma.
Then for any vertex in $\pwitv{i} \in \hat{T}_{n_0}$, the subtree rooted at $\pwitv{i}$ has linear growth almost surely, by the above coupling and Theorem~\ref{thm:PlambdaWITbis}.
Since this occurs for every vertex in $\hat{T}_{n_0}$, $\hat{T}$ itself exhibits linear growth almost surely, 
which completes the proof of Proposition~\ref{hatlinear}.

\section{Sharpness of the equivalence Theorem~\ref{thm:main}}\label{sec:examples}

In this section we give examples of pairs $(f,G)$ 
where the equivalence between $f$ being $G$-small and $f$ being $G$-explosive fails.  
We begin by giving simple examples for each direction of the equivalence.  

Note that this does not quite prove Proposition~\ref{prop:examples}, since 
the choice of $G$ used in these two simple examples differ. However, we will 
then prove Theorem~\ref{thm:lotsofcounters}, and since the choice of $G$ that we use in 
demonstrating Proposition~\ref{prop:examples}~(\ref{itm:simplenoexpl}) also satisfies the 
conditions of Theorem~\ref{thm:lotsofcounters}, Proposition~\ref{prop:examples} follows.

Recall that we write $F(n)$ for $\prod_{i=0}^{n-1}f(i)$, and note that $F(n)=|(T_f)_n|$.

\subsection{A pair $(f,G)$ where $f$ is $G$-small but not $G$-explosive}
We shall define a continuous \strictlyincreasing{} distribution function $G$ and a \nondecreasing{} function $f:\N_0 \to \N$, such that $f$ is $G$-small but not $G$-explosive.

In fact it is possible to define a general class of such examples.  Let $B_1, B_2, \dots$ 
be any \strictlyincreasing{} sequence of natural numbers.   
We may define $f$ such that the image of $f$ is precisely $\{B_i:i\ge 1\}$.

We first define $\G$, a 
distribution function that is neither continuous nor strictly increasing 
(in fact it is a \nondecreasing{} step function), with the property that $f$ is $\G$-small but not $\G$-explosive.  We then define $G$ as a tiny perturbation of $\G$ in such a way that $G$ becomes continuous and strictly increasing while maintaining the properties that $f$ is $G$-small but not $G$-explosive.

Our construction will involve defining a decreasing sequence $a_i$ of positive reals and a \nondecreasing{} sequence 
$n_i$ of natural numbers.  We will give an inductive construction of these sequences.

Set $a_0=1$ and $n_0=0$.  Given $a_{k-1}$ and $n_{k-1}$, consider the $B_k$-regular infinite tree $T_k$.  
Consider critical bond percolation on $T_k$ in which bonds are open with probability $1/B_k$.  
It is well known that all open clusters are finite almost surely.  
It follows that, for any constant $\eps>0$, there exists a constant $\xi(k,\eps)$ such that with probability at least 
$1-\eps$, every path from the root of $T_k$ to generation $\xi(k,\eps)$ uses at least $1/a_{k-1}$ closed edges.  
Set $n_k:=n_{k-1}+\xi(k,\eps_k)$, where $\eps_k:=\frac1{2F(n_{k-1})}$, and  
\[
f(n):=B_k \quad \text{for all} \quad n_{k-1}\le n<n_k\,. 
\]
Finally set $a_k:=\frac1{2^k n_k}$. This defines the sequences $a_i$, $n_i$, and the function $f: \mathbb N_0 \rightarrow \mathbb N$.

 Define now $\G$ by
\[
\G(x):=\frac 1{B_k}\quad \text{for all} \quad a_k \leq x < a_{k-1}
\]
It is easily checked that $f$ is $\G$-small.

\medskip

We now show that for each $k$ there is probability at least $1/2$ that every path from generation $n_{k-1}$ to 
generation $n_k$ has weight at least $1$.  
This is straightforward since $\xi(k,\eps)$ was precisely chosen so that with probability at 
least $1-\eps$ every path from a fixed vertex of generation $n_{k-1}$ to generation $n_k$ 
has at least $1/a_{k-1}$ edges of weight at least $a_{k-1}$, 
and therefore has total weight at least $1$.  
 The choice of $\epsilon_k$ is precisely made to 
guarantee that a union bound over vertices of generation $n_{k-1}$ complete the proof.

This property easily implies that explosion is a probability zero event, 
and so $f$ is not $\G$-explosive.

The essential content of the proof is unaffected if the atom of probability mass of $\G$ at $a_i$ is spread 
equally over the interval $[a_i,2a_i]$.  This makes the distribution function continuous. 
To make $G$ strictly increasing, add between $2a_i$ and $a_{i-1}$ probability mass with such a small 
total value that it is very unlikely that any edge in the first $n_i$ generations has a weight between 
$2a_i$ and $a_{i-1}$.  In this way the proof is again unaffected.

This defines a continuous and strictly increasing $G$ such that $f$ is $G$-small but not $G$-explosive.

\subsection{A pair $(f,G)$ where $f$ is $G$-explosive but not $G$-small}\label{sec:GexplnotGsmall}

We now define a continuous \strictlyincreasing{} distribution function $G$ and a \nondecreasing{} function 
$f:\N_0 \to \N$, such that $f$ is $G$-explosive but not $G$-small.

Again we may define a whole class of such examples.  Let $B_1,B_2,\dots$ be a sequence of natural numbers satisfying $B_{i+1}\ge 2 B_i$ for each $i\ge 1$.  We may define $f$ such that the image of $f$ is precisely $\{B_i:i\ge 1\}$.

As above, we first define an example in which $\G$ has atoms and then obtain $G$ using the same kind of perturbation trick as in the above construction.

Our construction will involve defining a decreasing sequence $a_i$ of positive reals and a \nondecreasing{} sequence $n_i$ of natural numbers.

Set $a_k:=\frac 1{k!}$ and $n_k:=(k-2)!$ for $k\ge 2$, $n_0=0$ and $n_1=1$.  Define $\G$ by
\[
\G(x):= \frac{1-1/(k-1)!}{B_k} \quad \text{for all } a_k \leq x < a_{k-1},
\]
and note that $\G^{-1}(1/B_k) = a_{k-1}$. 
Define $f:\N_0\to \N$ by
\[
f(n):=B_k \quad \text{for all } n_{k-1}\le n<n_k\, . 
\]
It is easily checked that $f$ is not $\G$-small.

To see that $f$ is $\G$ explosive with positive probability 
(and therefore with probability $1$ by the $0$-$1$ law) 
consider the following strategy for finding an infinite path of finite weight. 
Define a forest $\mathfrak{T}\subseteq T^{\G}_f$ 
by keeping the following edges: between generations $n_k=(k-2)!$ and $2(k-2)!$ keep all edges of 
weight at most $a_{k}=1/k!$, and between generations $2(k-2)!$ and $n_{k+1}=(k-1)!$ keep all 
edges of weight at most $a_{k+1}=1/(k+1)!$.  It is clear that any infinite path in 
$\mathfrak{T}$ necessarily has finite weight.  
Writing $p_k$ for the probability that a node of generation $n_k$ has no descendent in 
$(\mathfrak{T})_{n_{k+1}}$, it suffices to prove that $\sum_k p_k<\infty$.  

Let us bound $p_k$ from above.  Let $v$ be a node of $\mathfrak{T}_{n_k}$ and consider the 
tree of descendants of $v$ in $\mathfrak{T}$.  In the first $(k-2)!$ generations following $n_k$, 
each node has $\Bin(B_{k+1},\frac{1-1/(k-1)!}{B_k}))$ children.  
Denoting by $Z(.)$ the branching process with 
offspring distribution $\Bin(B_{k+1},\frac{1-1/(k-1)!}{B_k}))$, which has mean at least $3/2$, 
it is straightforward to verify that the probabilities
\[
q_k\, :=\, \Pr{ Z\bigl((k-2)!\bigr) <\left(\frac{4}{3}\right)^{(k-2)!}}
\]
are summable.  In addition, let us note that each node $u\in \mathfrak{T}_{2(k-2)!}$ 
has probability at least $(1-1/k!)^{(k-1)!}\ge e^{-1}$ to have a descendent in generation 
$n_{k+1}=(k-1)!$.  The probability $r_k$ that none of $(4/3)^{(k-2)!}$ nodes of generation 
$2(k-2)!$ has a descendent in generation $n_{k+1}$ is at most $e^{-(4/3)^{(k-2)!}}$, 
which is clearly summable.

The proof is now complete since $p_k\le q_k+r_k$.

\subsection{Proof of Theorem~\ref{thm:lotsofcounters}}
Let $G$ be a distribution function satisfying~\eqref{cond1}, i.e., such that
\[
\limsup_{i\to \infty} \frac{G(x_i)}{G(x_i/c)}\, < \, \limsup_{i\to \infty} \frac{G(cx_i)}{G(x_i)}\, =\,  \infty\, ,
\]
for some constant $c>0$ and a decreasing sequence $x_i:i\ge 1$ with limit $0$.  Restricting to a subsequence if necessary, we may assume that
\[
\frac{G(x_i)}{G(x_i/c)}\, \le \, K \, \le \, K 4^{4^i}\, \le \,  \frac{G(cx_i)}{G(x_i)} \, ,
\]
where $K$ is a constant, and $x_{i-1}\ge 4^i x_i$ for each $i\ge 1$.  We may also assume that $G(x_i)\le 1/4$.

One may now define the sequence $n_i$ by $n_0:=0$ and $n_i:= \lfloor 1/x_i\rfloor$ for $i\ge 1$, and the function $f:\N_0\to \N$ by
\[
f(n)\, :=\, \left\lfloor\frac{1}{2G(x_i)}\right\rfloor \qquad \text{for } n_{i-1}<n\le n_i\, .
\]
This choice of $f(n)$ is designed to be around $1/G(x_i)$, note that it satisfies
\[
\frac{1}{4G(x_i)}\, < \, f(n)\, <\, \frac{1}{G(x_i)}\, .
\]
Using the second inequality above and the choice of the sequence $n_i$ it is easily observed that $f$ is not $G$-small. 

Consider the forest $\mathfrak{T}\subseteq T_f^G$ obtained by keeping edges between 
generation $n_{i-1}$ and $n_{i-1}+n_i/2^i$ with weight at most $cx_i$, 
and edges between generations $n_{i-1}+n_i/{2^i}$ and $n_i$ with weight at most $x_i/ 2^{i}$.  
Clearly any infinite path in $\mathfrak{T}$ has weight at 
most $\sum_{i\ge 1}2^{-i}+\sum_{i\ge 1}2^{-i}<\infty$.  Thus, to prove that $f$ is $G$-explosive 
it suffices to prove that $\mathfrak{T}$ contains an infinite path with positive probability.
The idea of the proof is that with very high probability a node of generation $n_{i-1}$ will have a very large number of descendants in generation $n_{i-1}+n_i/{2^i}$, and each such node has a not so small probability of having a descendant in generation $n_i$.
While we do not go through every detail of the proof, it suffices to
note that between generations $n_{i-1}$ and $n_{i-1}+n_i/{2^i}$ the number of children of
each node stochastically dominates
\[
\Bin\left(f(n),K4^{4^i}G(x_i)\right)\, \ge \, \Bin\left(f(n),\frac{K4^{4^i -1}}{f(n)}\right)\, ,
\]
which has mean $K4^{4^i-1}$.  
It is therefore extremely likely that a node $v$ of generation $n_{i-1}$ has at least $e^{2^{i}n_i}$ 
descendants in generation $n_{i-1}+n_i/2^i$, denote this event $E_v$.

Between generations $n_{i-1}+n_i/{2^i}$ and $n_i$ each node has at least one child with probability at
least $K^{-2\lceil \log_c (2^i)\rceil}\ge e^{-K' i}$ for some constant $K'$.  
Therefore a node of generation $n_{i-1}+n_i/2^i$ has a descendant in generation $n_i$ 
with probability at least $e^{-K'in_i}$.
Thus, the number of descendants in generation $n_{i}$ of a node $v$ of generation $n_{i-1}$ stochastically dominates
\[
1_{E_v} \cdot \Bin((e^{2^{i}n_i},e^{-K'in_i})
\]
which is zero with very small probability.

Up to routine details this completes the proof that $\mathfrak{T}$ contains an infinite path 
with positive probability.

\medskip

The proof in the case that $G$ satisfies \eqref{cond2} is similar to the proof given in 
Section~\ref{sec:GexplnotGsmall} -- 
define a forest $\mathfrak{T}$ which grows exponentially for a period after 
generation $n_{i-1}$, then only slightly sub-critically from there until generation $n_i$.  
We omit the details.

\section{Finite height criterion for stick breaking problem}
\label{sec:stick}

Recall the definition of the random real tree $A_\ell$ constructed by a
stick breaking process on $\Rplus$ given by a sequence  $\ell(i),$ for $i\in \mathbb{N}$:
Given such a sequence, the stick breaking process defines a 
random real 
tree $A_\ell$ as follows.  Let $A_{\ell}(1)$ consist of a closed segment of length $\ell(1)$, 
seen as a (rooted) real 
tree with one edge, rooted at one end, and 
for each $i\ge 1$ let $A_{\ell}(i+1)$ be obtained by attaching one end of a closed segment
of length $\ell(i+1)$ to a uniformly random position on the real tree $A_{\ell}(i)$. 
Define 
$A_\ell$ as the completion of  $A_{\ell}^o\, =\, \bigcup_{i\ge 1}A_{\ell}(i)\, .$

For any real tree $A$, denote by $d(A)$ the height of $A$. 
Our aim in this section is to prove Theorem~\ref{thm:stick}, namely, to show that if $\ell$ 
is a decreasing sequence, then 
 \[d(A_{\ell})<\infty \quad \textrm{almost surely if and only if} \quad 
\sum_{n\ge 1}\ell(2^n)\, <\, \infty\, .
\]

The intuition behind the summability condition is roughly speaking as follows: 
traversing a path through $A_{\ell}$ the index of the segment which contains 
the $n$th edge used on the path should grow exponentially in $n$, so that the sum of the
lengths of the segments containing the edges of the path should behave like
\[
\sum_{n\ge 1}\ell(2^n)\, .
\]

The following observation indeed allows us to focus on the sum of the lengths of the 
segments (i.e., $\ell(i)$s) rather than distances in the real tree.  
Given a path $\xi$ in $A_\ell$, let $td(\xi)$ denote the sum 
of the lengths of the segments which contain the edges of $\xi$. 
Let $td(A_\ell)$ be the maximum of $td(\xi)$ over all the paths in the rooted tree $A_\ell$.

\begin{observation}\label{observation} The inequality $d(A_{\ell})\, \le\, td(A_{\ell})$ holds deterministically.  On the other hand, if $td(A_{\ell})$ is infinite then $d(A_{\ell})$ is infinite almost surely.
\end{observation}

The observation is indeed straightforward as the randomness used to position the 
endpoint of a segment $e_j$ on a current segment $e_i$ may be sampled independently.  
In the limit at least half of the edges connect at least half way along. 

\medskip

The following two lemmas will be useful in our proof of the theorem.  The first is a straightforward monotonicity statement.  The second will be used to show that when $\sum_{n\ge 1}\ell(2^n)$ is divergent, it cannot be that $\ell(i)$ is always much smaller than the average of $\ell(j):j\le i$.

\begin{lemma}\label{lem:mono}
Suppose $R,S$ are two subsets of $\mathbb{N}$ and for all $j\in \mathbb N$,
\[
|R\cap\{1,\dots ,j\}|\, \le\, |S\cap \{1,\dots ,j\}|.
\]
Then
\[
\sum_{i\in R}\ell(i)\, \le\, \sum_{i\in S} \ell(i) \, .
\]
\end{lemma}

\begin{proof}
The sums are limits of the partial sums up to $j$.  
For any fixed $j$, we can find an injection $\mathfrak i: R\cap\{1,\dots ,j\} \hookrightarrow
S\cap\{1,\dots ,j\}$ 
such that $\mathfrak i(i) \leq i$ for any $i\in R\cap\{1,\dots ,j\}$. 
Since $\ell$ is decreasing, we get 
$$\sum_{i\in R\cap\{1,.\dots, j\}} \ell(i) \leq \sum_{i\in R\cap\{1,.\dots, j\}}\ell(\mathfrak i(i)) 
\leq \sum_{i\in S\cap\{1,.\dots, j\}}\ell(i)$$
for partial sums, from which the lemma follows.
\end{proof}

\begin{lemma}\label{lem:D}
Let $\ell(i)$ be an integer sequence and let 
\[
\D\, :=\, \left\{n\in \mathbb{N}: \, \, \ell(2^n)\, > \, \frac{\ell(2^m)}{2^{(n-m)/2}}\quad  \text{for all } \, m< n\right\}\,.
\]
Then 
\[\sum_{n\ge 1}\ell(2^n)\,\,\, \textrm{is divergent} \,\, \textrm{if and only if} \,\, 
\sum_{n\in \D}\ell(2^n)\,\,\, \textrm{is divergent}.
\]
\end{lemma}

\begin{proof} For each $n\in \mathbb{N}$ define $\pi(n)$ to be the least natural number $m$ such that
\[
\ell(2^n)\, \le \, \frac{\ell(2^m)}{2^{(n-m)/2}}\, .
\]
Note that $n\in \D$ if and only if $\pi(n)=n$, and $\pi(\pi(n)) = \pi(n)$ for any $n\in \mathbb N$,
in other words, $\pi$ is a projection from 
$\mathbb N$ to $\D$.
For each $m\in \D$, let 
\[
H_m \, :=\, \{n:\pi(n)=m\} \, .
\]
The lemma follows immediately from the observations that
\[
\mathbb{N}\, =\, \bigcup_{m\in \D}H_m
\]
and that for any $m\in \D$, 
\[
\sum_{n\in H_m}\ell(2^n)\, \leq \, \sum_{n=m}^{\infty} \frac{\ell(2^m)}{2^{(n-m)/2}}\, \le\, 4\ell(2^m)\, .
\]
\end{proof}

We are now ready to prove Theorem~\ref{thm:stick}.

\begin{proof}
Neither direction is trivial.  We begin by showing that if
\[
\sum_{n\ge 1}\ell(2^n)\, <\, \infty\,,
\]
then $d(A_\ell)$ is finite almost surely.  Let us identify a path 
$\xi$ in $A_\ell$ with the set of segments it uses, and further, identify $\xi$ with the subset 
of $\mathbb{N}$ of indices of these segments.  The length of the path $\xi$ is obviously bounded 
by $\sum_{i\in \xi}\ell(i)$.  By Lemma~\ref{lem:mono} it 
suffices to show that every path $\xi$ in $A_{\ell}$ verifies 
\[
|\xi\cap\{1,\dots ,j\}|\, \le\, |R\cap \{1,\dots ,j\}| \qquad \text{for all } j\in \mathbb{N}\, .
\]
for some set $R$ of the form $R=\{\lfloor e^{\delta n}\rfloor :n\ge 1\}$.

We prove this using our result, Theorem~\ref{thm:PlambdaWITbis}, 
on linear growth in generalizations of the PWIT.  
We can construct an infinite weighted (random) tree $T$ associated with $A_{\ell}$. As a (combinatorial) 
tree, $T$ is the genealogy tree of $A_\ell$, formally defined as follows: 
To each segment $e_i$ (of length $\ell(i)$) used in $A_{\ell}$ we associate a vertex 
$v_i$ of $T$.  The vertex $v_1$ will be the root of $T$, and the vertex $v_i$ is a child of 
another vertex 
$v_j$ if $j<i$ and in the construction of $A_\ell$, the segment $e_i$ is attached to a point of 
$e_j$.  The weight of the edge $v_jv_i$ is defined to be $\log{j}-\log{i}$.  For technical reasons we shall ignore the root and the edges to the children of the root.  Consider the subtree beneath some vertex $v_i$ that is a child of the root.  The probability that $v_j$ will be a child of $v_i$ for $j>i$ is at most $1/i$, therefore the distribution of the number of children of $v_i$ of weight at most $w>0$ is stochastically dominated by the binomial distribution
\[
\Bin(\lfloor e^{w}i-i\rfloor ,1/i)\, =\, \Bin(\lfloor i(e^w -1)\rfloor ,1/i)\, \le\, \Po\left(\int_0^w C^t \, dt\right)\,
\]
for an appropriately chosen constant $C>1$ (in fact one may take $C=e^2$).
Theorem~\ref{thm:PlambdaWITbis} now tells us that all subtrees of the children of the root exhibit linear growth, and furthermore the probability of a path to generation $n$ of weight less that $\delta n$ (for some $\delta>0$) is at most $e^{-2n}$.  The same is true deterministically for paths beginning with an edge of weight at least $n$.  Since at most $e^n$ children of the root have weight less than $n$, a union bound yields that
\[
m_n(T)\, \ge \, \delta n
\]
with probability at least $1-e^{-n}$.  And so, almost surely, $m_n(T)\ge \delta n$ for all sufficiently large $n$

This translates into the fact that every path $\xi$ in $A_\ell$, when viewed as a subset of $\mathbb{N}$ has $n$th element at least $e^{\delta n}$ for all sufficiently large $n$.  This completes the required comparison with a set $R$, and so completes this half of the proof.

\medskip

For the other direction, let us assume that the sum
\[
\sum_{n\ge 1}\ell(2^n)
\]
is divergent.  Applying Lemma~\ref{lem:D}, we have that
\[
\sum_{n\in \D}\ell(2^n)\, 
\]
is divergent, where
\[
\D\, :=\, \left\{n\in \mathbb{N}: \ell(2^n)\, \ge \, \frac{\ell(2^m)}{2^{(n-m)/2}}\, \text{for all } m< n\right\}\, .
\]
Let $\overline{\D}$ denote the set $\bigcup_{n\in \D}\{2^{n-1}+1,\dots ,2^n\}$, 
and consider the path $\xi$ in $A_{\ell}$ generated as follows: the first segment 
of $\xi$ is $e_1$, thereafter $\xi$ chooses to connect to the segment $e_i$ with minimal index 
$i\in \overline{\D}$ that is a descendant of its latest segment.  
By Observation~\ref{observation}, 
we complete a proof of the theorem by proving that $td(\xi)$ is 
infinite almost surely.

Note that for a sequence of non-negative reals $\bigl(a(i)\bigr)$, if the sum $\sum_{i\in \mathbb N}a(i)$ 
diverges, then a sum of the form 
$\sum_{i\in S}a(i)$ for a random subset $S\subset \mathbb N$ which contains 
each index $n\in N$ independently with probability $p>0$, will diverge almost surely.  
Similarly, if there is not necessarily independence between the inclusion of indices in $S$ 
but $S$ 
is constructed recursively 
so that the inclusion of $n$ in $S$ (given all the previous information) 
has probability at least $p>0$, then again the sum $\sum_{i \in S}a(i)$ will diverge almost surely.  

\noindent For this reason, it suffices to show that with probability at least some constant $p>0$ 
the path $\xi$ will contain a segment $e_i$ with $2^{n-1}+1\le i\le 2^n$. 

Given the path $\xi$ up to inclusion of some edge $e_j:j\in \overline{\D}$, let $m\in \D$ be 
the integer such that $2^{m-1}<j\le 2^m$, and let $n\in \D$ be minimum integer in $\D$ 
such that $2^{n-1}\ge j$. We complete the proof by bounding below the probability that $\xi$ 
contains a segment $e_i$ with $2^{n-1}+1\le i\le 2^n$.  Until $2^n$ segments have been attached in the construction of $A_\ell$, 
the total length may be bounded above by
\[
\sum_{k=1}^{n}2^{k-1}\ell(2^{k-1})\, \le\, \sum_{k=1}^{n}2^{k-1}2^{(n-k)/2}\ell(2^{n})\, \le \, 2^{n+2}\ell(2^n)\, \le \, 2^{n+2}\ell(j)\, .
\]
And so, each segment added has probability at least $2^{-n-2}$ of being attached to $e_j$.  
The probability that no segment $e_i$ with $2^{n-1}<i\le 2^{n}$ joins to $e_j$ is 
therefore at most
\[
\left(1-\frac{1}{2^{n+2}}\right)^{2^{n-1}}\, \leq\, e^{-1/8}\, .
\]
This completes the proof.
\end{proof}

\medskip

\subsection*{Acknowledgements}
This project grew out of discussions started at the McGill University's Bellairs Institute, Barbados. 
We would like to thank Nicolas Curien for asking 
a question which led to the results in 
Section~\ref{sec:stick}, and thank B\'en\'edicte Haas and him for sharing a draft of their work 
at the final stage of the preparation of this paper.
Special thanks to the Brazilian-French Network in Mathematics 
for providing generous support for a visit of S.G. at ENS Paris. 
N.O. was supported by a NWO Veni grant.


\begin{thebibliography}{WWW}

\bibitem{AdGK} Addario-Berry, L., Griffiths, S., and Kang, R.~J. (2012). Invasion percolation on the Poisson-weighted infinite tree. 
\textsl{Ann.\ Appl.\ Probab.} {\bf 22}(3), 931–-970. 

\bibitem{AdRe09} Addario-Berry, L. and Reed, B. (2009). Minima in branching random walks, \textsl{Ann.\ Appl.\ Probab.} 
{\bf37}, 1044-1079.

\bibitem{Ai11} A\"idekon, E. (2013). Convergence in law of the minimum of a branching random walk. 
\textsl{Ann.\ Probab.} {\bf 41}(3A), 1362-1426.

\bibitem{AiSh11} A\"idekon, E. and Shi, Z. (2014). The Seneta-Heyde scaling for the branching random walk. 
\textsl{Ann.\ Probab.} {\bf 42}(3), 959-993. 

\bibitem{Ald91} Aldous, D.~J. (1991). The continuum random tree I. \textsl{Ann. Probab.} {\bf 19}, 1--28.

\bibitem{Ald92} Aldous, D.~J. (1992). Asymptotics in the random assignment problem. \textsl{Probab.\
Th.\ Rel.\ Fields} {\bf 93}, 507--534.

\bibitem{Ald93} Aldous, D.J. (1993). The continuum random tree III. \textsl{Ann. Probab.} {\bf 21}, 
248--289.


\bibitem{AlPi00} Aldous, D.~J. and Pitman J. (2000). Inhomogeneous continuum random 
trees and the entrance boundary of the additive coalescent. \textsl{Probab. Th. Rel. Fields} {\bf 118}, 
455--482.

\bibitem{AS} Aldous, D.~J.
and Steele, J.~M. (2003). The objective method: Probabilistic combinatorial optimization and local weak convergence. In
\textsl{Probability on Discrete
Structures}, Encyclopedia of Mathematical Sciences (H. Kesten, ed.) {\bf 110} 1--72. Springer, Berlin.

\bibitem{ADGO} Amini, O., Devroye, L., Griffiths, S., and Olver, N. (2013).  
On explosions in heavy-tailed branching random walks. {\it Ann.\ Probab.} {\bf 41}(3B), 1864--1899.

\bibitem{Big76} Biggins, J.~D. (1976) 
The first- and last-birth problems for a multitype age-dependent branching process. \textsl{Adv. Appl. Probab.} {\bf 8}, 446--459.

\bibitem{Big77} Biggins, J.~D. (1977). Chernoff's Theorem in the branching random walk. 
\textsl{J.\ App.\ Prob.} {\bf 14}(3), 630-636.


\bibitem{BCC} Bordenave, C., Caputo, C., and Chafa\"i, D. (2011). Spectrum of large random reversible Markov chains: heavy-tailed weights on the complete graph. \textsl{Ann.\ Probab.} 
{\bf 39}(4), 1544--1590.

\bibitem{Bra78} Bramson, M.~D. (1978). Minimal displacement of branching random walks. \
textsl{Zeitschrift f\"ur Wahrscheinlichkeitstheorie und verwandte Gebiete} \textbf{45}, 89--108.   

\bibitem{Curien} Curien, N. (2014). Open question at the Probability, Combinatorics and Geometry workshop held at McGill University's Bellairs Institute, Barbados.

\bibitem{CH} Curien, N. and Haas B., Random trees constructed by aggregation. Preprint.

\bibitem{DeHo91} Dekking, F.~M. and Host, B. (1991). Limit distributions for minimal displacement of
branching random walks. \textsl{Probab.\ Th.\ Rel.\ Fields} \textbf{90}, 403--426.

\bibitem{GH14} Goldschmidt C. and Haas B. (2014). 
A line-breaking construction of the stable trees. \textsl{Available at http://arxiv.org/abs/1407.5691} 

\bibitem{Ham74} Hammersley, J.~M. (1974). Postulates for subadditive processes. \textsl{Ann.\ Probab.} \textbf{2}, 652--680. 

\bibitem{HuSh09} Hu, Y. and Shi, Z. (2009). Minimal position and critical martingale convergence
in branching random walks, and directed polymers on disordered trees. \textsl{Ann.\ Probab.} \textbf{37},  742--789.

\bibitem{Kin75} Kingman, J.~F. (1975). The first birth problem for an age-dependent branching process. 
\textsl{Ann.\ Probab.} \textbf{3},  790--801.
 
 \bibitem{McD95} McDiarmid, C. (1995). Minimal positions in a branching random walk. \textsl{Ann.\ Appl.\ Probab.} 
 \textbf{5}, 128--139.

 \bibitem{PP} Pemantle, R. and Peres, Y. (1994). Domination between trees and application to an explosion problem. 
 {\it Ann.\ Probab.} {\bf 22}(1), 180--194.


\bibitem{Vat96} Vatutin, V.~A. (1996). On the explosiveness of nonhomogeneous age-dependent
branching processes. \textsl{Theory Probab.\ Math.\ Statist.} {\bf 52},  39--42; translated from Teor. Imovir. Mat. Stat. 
No. 52 (1995), 37--40 (Ukrainian).

 
 \bibitem{VZ93} Vatutin, V.~A. and Zubkov, A.~M. (1993). Branching processes II.  Probability theory and mathematical statistics, 1. 
\textsl{ J.\ Soviet Math.}  {\bf 67}, 3407--3485.
    
\end{thebibliography}
    \end{document}